\documentclass[11pt]{article}
\usepackage[ansinew]{inputenc}
\usepackage{graphicx}
\usepackage{color}

\usepackage{enumerate,latexsym}
\usepackage{latexsym}
\usepackage{amsmath,amssymb}% CUIDADO: Si uso estos packages, $\widetilde{}$ no funciona a veces
\usepackage{graphicx}
\newfont{\bb}{msbm10 at 11pt}
\newfont{\bbsmall}{msbm8 at 8pt}
\def\rth{\mathbb{R}^3}
\def\R{\mathbb{R}}
\def\B{\mathbb{B}}
\def\N{\mathbb{N}}

\def\C{\mathbb{C}}

\def\D{\mathbb{D}}

\def\esf{\mathbb{S}}
\def\Te{\mathbb{T}}

\def\cL{\mathcal{L}}

\newcommand{\ben}{\begin{enumerate}}
\newcommand{\bit}{\begin{itemize}}
\newcommand{\een}{\end{enumerate}}
\newcommand{\eit}{\end{itemize}}
\newcommand{\wh}{\widehat}

\newcommand{\wt}{\widetilde}

\def\lc{{\cal L}}

\def\g{{\gamma}}

\def\l{{\lambda}}

\def\de{{\delta}}
\def\be{{\beta}}
\def\ve{{\varepsilon}}

\def\centerbmp#1#2#3{\vskip#2\relax\centerline{\hbox to#1{\special
    {bmp:#3 x=#1, y=#2}\hfil}}}

\newtheorem{theorem}{Theorem}[section]
\newtheorem{lemma}[theorem]{Lemma}
\newtheorem{proposition}[theorem]{Proposition}

\newtheorem{remark}[theorem]{Remark}

\newtheorem{definition}[theorem]{Definition}

\newenvironment{proof}{\smallskip\noindent{\it Proof.}\hskip \labelsep}
{\hfill\penalty10000\raisebox{-.09em}{$\Box$}\par\medskip}

\begin{document}
\begin{title}
{The Dynamics Theorem for properly embedded minimal surfaces}
\end{title}
\vskip .2in

\begin{author}
{William H. Meeks III\thanks{This material is based upon
   work for the NSF under Award No. DMS - 1309236.
   Any opinions, findings, and conclusions or recommendations
   expressed in this publication are those of the authors and do not
   necessarily reflect the views of the NSF.}
   \and Joaqu\'\i n P\' erez
\and Antonio Ros\thanks{The second and third authors were supported in part
by the MEC/FEDER grant no. MTM2011-22547, and by the
regional J. Andaluc\'\i a grant no. P06-FQM-01642.}}
\end{author}
\maketitle
\begin{abstract}
In this paper we prove two theorems. The first one is a structure result
 that describes the extrinsic geometry of
an embedded surface with constant mean curvature (possibly zero)
in a homogeneously regular Riemannian three-manifold, in any small neighborhood
of a point of large {\em almost-maximal curvature.} We next apply this theorem and the
Quadratic Curvature Decay Theorem in~\cite{mpr10}
%of the complete, embedded minimal surfaces in $\R^3$
%with quadratic decay of curvature as those examples having finite total curvature
to deduce compactness, descriptive and {\it dynamics-type} results concerning
the space $D(M)$ of non-flat limits under dilations of any given properly embedded minimal
surface $M$ in $\rth$.
\vspace{.3cm}


\end{abstract}

\section{Introduction.}
In this paper we will describe the local extrinsic geometry of a complete embedded surface
$M$ of constant mean curvature around a point of large Gaussian curvature in a homogeneously regular
Riemannian three-manifold (see Theorem~\ref{thm3introd}). We then obtain some
consequences among which we highlight the {\it Dynamics Theorem} (Theorem~\ref{thm4introd})
concerning the space $D(M)$ of non-flat limits under a divergent sequence of dilations of any given properly
embedded minimal surface $M$ in $\rth$; by a dilation we mean any diffeomorphism of $\R^3$ into
itself given by composition of a translation with a homothety.
An important consequence of the Dynamics Theorem
is that every properly embedded minimal surface in $\R^3$ with infinite total curvature has
a surprising amount of internal dynamical periodicity;
see Theorem~\ref{thm4introd} and Proposition~\ref{propos4.1} below for
this interpretive consequence.

In order to state the main results, it is
worth setting some specific notation to be used throughout the paper.
Given a Riemannian three-manifold $N$ and a point $p\in N$, we  denote
by $B_N(p,r), \overline{B}_N(p,r)$, $S_N^2(p,r)$ the open metric ball of center $p$
and radius $r>0$, its closure and boundary sphere, respectively.
In the case $N=\R^3$, we use the notation $\B (p,r)=B _{\R^3}(p,r)$, $\esf^2(p,r)=
S^2_{\R^3}(p,r)$ and $\B (r)=\B (\vec{0},r)$, $\esf^2(r)=\esf^2(\vec{0},r)$,
where $\vec{0}=(0,0,0)$.  For a surface $M$ embedded
in $N$, we denote by $|\sigma _{M}|$ the norm of the second fundamental form of $M$.
Finally, for a surface $M\subset \R^3$, $K_M$ denotes its Gaussian curvature function.
We will call a Riemannian three-manifold
$N$ {\it homogeneously regular} if there exists
an $\ve > 0$ such that $\ve$-balls in $N$ are uniformly close to
$\ve$-balls in $\rth$ in the $C^2$-norm. In particular, if $N$ is
compact, then $N$ is homogeneously regular.

\begin{theorem} [Local Picture on the Scale of Curvature]
\label{thm3introd}
Suppose $M$ is a complete, embedded, constant
mean curvature surface (here we include minimal
surfaces as being those with constant mean curvature zero)
with unbounded second fundamental form
in a homogeneously regular
three-manifold $N$.
Then, there exists a sequence of points
$p_n\in M$ and positive numbers $\ve _n\to 0$, such that the
following statements hold.
\begin{enumerate}
\item  For all $n$, the component $M_n$ of
$\overline{B}_N(p_n,\ve _n)\cap M$ that
contains $p_n$ is compact with boundary $\partial M_n\subset
\partial B_N(p_n,\ve _n)$.
\item Let $\lambda_{n} =|\sigma _{M_{n}}(p_{n})|$. Then,
$\lim_{n\to \infty }\ve_n\l_n=\infty $ and
$\frac{|\sigma _{M_n}|}{\l _n}\leq 1+\frac{1}{n}$ on $M_n$.
.
\item The metric balls $\l _nB_N(p_n,\ve _n)$ of radius $\l _n\ve _n$
converge uniformly to $\R^3$ with its usual metric (so that we
identify $p_n$ with $\vec{0}$ for all $n$), and, for any $k \in \N$,
the surfaces $\l _nM_n$ converge $C^k$ on compact subsets of $\rth$
and with multiplicity one to a connected, properly embedded minimal
surface $M_{\infty}$ in $\R^3$ with $\vec{0}\in M_{\infty }$,
$|\sigma _{M_{\infty}}|\leq 1$ on $M_{\infty}$ and
$|\sigma _{M_{\infty}}|(\vec{0})=1$.
\end{enumerate}
\end{theorem}

Every complete, embedded minimal surface in $\R^3$
with bounded curvature is properly embedded, see
Meeks and Rosenberg~\cite{mr8}. The key idea in the proof of Theorem~\ref{thm3introd}
is to exploit this fact, together with a careful blow-up argument.
Note that a direct consequence of Theorem~\ref{thm3introd} is that every
complete embedded surface with constant mean curvature in $\R^3$ which is not
properly embedded, has natural limits under a sequence of dilations,
which are properly embedded non-flat minimal surfaces.

Theorem~\ref{thm3introd} together with the
Local Removable Singularity Theorem in~\cite{mpr10}
can be used  to understand the structure of the collection of
limits of a non-flat, properly embedded minimal surface in $\R^3$
under any divergent sequence of dilations, which is the purpose of the next
theorem. In order to clarify its statement, we need some definitions.
\begin{definition}
\label{def1.8}
\rm{ Let $M\subset \R^3$ be a non-flat, properly
embedded minimal surface. Then:
\begin{enumerate}
\item $M$ is {\it periodic} if it is invariant under a non-trivial
translation or screw motion symmetry.

\item $M$ is {\it quasi-translation-periodic} if there exists a divergent
sequence $\{ p_n\} _n\subset \R^3$ such that $\{ M-p_n\} _n$
converges $C^2$ on compact subsets of $\R^3$ to $M$; note that every
periodic surface is also quasi-translation-periodic, even in the
case the surface is invariant under a screw motion symmetry.

\item $M$ is {\it quasi-dilation-periodic} if there exists a sequence $\{ \l _n\} _n
\subset \R ^+$ and a divergent sequence $\{ p _n\}
_n\subset \R ^3$ such that the sequence $\{ \l _n(M-p_n)\} _n$ of dilations of $M$ converges in a
$C^2$-manner on compact subsets of $\R^3$ to $M$. Since $M$ is not
flat, then it is not stable and, it can be proved that the convergence of such a sequence
$\{ h_n(M-p_n)\} _n$ to $M$ has multiplicity one (see Lemma~\ref{lemma4.2} below for a similar
result about multiplicity one convergence).

\item Let $D(M)$ be the set of non-flat, properly
embedded minimal surfaces in $\R^3$ which are
obtained as $C^2$-limits
of a divergent sequence $\{ \l _n(M-p_n)\} _n$ of dilations of $M$
(i.e., the translational part $\{ p_n\} _n$ of the dilations diverges). A non-empty
subset $\Delta \subset D(M)$ is called {\it $D$-invariant,} if for
any $\Sigma \in \Delta $, then $D(\Sigma ) \subset \Delta $. A
$D$-invariant subset $\Delta \subset D(M)$ is called a {\it minimal
$D$-invariant set,} if it contains no proper, non-empty
$D$-invariant subsets. We say that $\Sigma \in D(M)$ is a {\it
minimal element,} if $\Sigma $ is an element of a minimal
$D$-invariant subset of $D(M)$.
\end{enumerate}}
\end{definition}

\begin{theorem} [Dynamics Theorem]
\label{thm4introd}
Let $M\subset \R^3$ be a properly embedded, non-flat
minimal surface. Then, $D(M)=\mbox{\rm \O}$
if and only if $M$ has finite total curvature. Now assume that $M$
has infinite total curvature, and consider $D(M)$ endowed with the
topology of $C^k$-convergence on compact sets of $\R^3$ for all $k$.
Then:
\begin{enumerate}
\item $D_1(M)=\{ \Sigma \in D(M)\ | \ \vec{0}\in \Sigma ,\ |K_{\Sigma }
|\leq 1,\ |K_{\Sigma }|(\vec{0})=1\}$ is a non-empty compact subspace of $D(M)$.
\item For any $\Sigma \in D(M)$, $D(\Sigma )$ is a closed $D$-invariant set of
$D(M)$. If $\Delta \subset D(M)$ is a $D$-invariant set, then its
closure $\overline{\Delta}$ in $D(M)$ is also $D$-invariant.
\item Suppose that $\Delta \subset D(M)$ is a non-empty minimal $D$-invariant set
which does not consist of exactly one surface of finite total
curvature. If $\Sigma \in \Delta$, then $D(\Sigma)=\Delta$ and the
closure in $D(M)$ of the path connected subspace
$\{ \l (\Sigma -p)\ | \ p\in \R^3, \, \l >0\} $ of all dilations of
$\Sigma$ equals $\Delta$. In particular, any minimal $D$-invariant
set is connected and closed in $D(M)$.
\item Every non-empty $D$-invariant subset of $D(M)$ contains minimal
elements. In particular, since $D(M)$ is $D$-invariant, then $D(M)$
always contains some minimal element.
\item Let $\Delta \subset D(M)$ be a non-empty $D$-invariant subset. If no
$\Sigma \in \Delta $ has finite total curvature, then $\Delta _1=\{
\Sigma \in \Delta\ | \ \vec{0}\in \Sigma ,\ |K_{\Sigma } |\leq 1,\
|K_{\Sigma }|(\vec{0})=1\} $ contains a minimal element $\Sigma '$,
and every such surface $\Sigma '$ satisfies that
$\Sigma '\in D(\Sigma ')$ (in particular, $\Sigma '$ is a
quasi-dilation-periodic surface of bounded curvature).
\item If a minimal element $\Sigma $ of $D(M)$ has finite genus,
then either $\Sigma $ has finite total curvature,  $\Sigma $ is a
helicoid, or $\Sigma $ is a Riemann minimal example.
\item If a minimal element $\Sigma $ of $D(M)$ has more than one end,
then every middle end of $\Sigma $ is smoothly
asymptotic to the end of a plane or catenoid.
%\item If $\Sigma $ is a minimal element of $D(M)$ and $D_1(\Sigma)$ contains a surface with more than one end. Then
% %every element of $D(\Sigma)$ has bounded curvature with two limit ends.
% all of the surfaces in $D(\Sigma)$ are  quasi-translation-periodic.
\end{enumerate}
\end{theorem}

\section{Preliminaries.}
When proving the results stated in the introduction, we will make use of three results from our previous paper~\cite{mpr10}. For the reader's convenience,
we collect here these results and the definitions necessary in order to state them.

\begin{definition}
\label{deflamination}
{\rm
 A {\it codimension one
lamination} $\cL$ of a Riemannian three-manifold $N$ is the union of a
collection of pairwise disjoint, connected, injectively immersed
surfaces called leaves of $\cL$, with a certain local product structure. More precisely, it
is a pair $({\mathcal L},{\mathcal A})$ satisfying:
\begin{enumerate}
\item ${\mathcal L}$ is a closed subset of $N$;
\item ${\mathcal A}=\{ \varphi _{\be }\colon \D \times (0,1)\to
U_{\be }\} _{\be }$ is an atlas of coordinate charts of $N$ (here
$\D $ is the open unit disk in $\R^2$, $(0,1)$ is the open unit
interval and $U_{\be }$ is an open subset of $N$); note that
although $N$ is assumed to be smooth, we
only require that the regularity of the atlas (i.e., that of
its change of coordinates) is of class $C^0$, i.e., ${\cal A}$
is an atlas for the topological structure of $N$.
\item For each $\be $, there exists a closed subset $C_{\be }$ of
$(0,1)$ such that $\varphi _{\be }^{-1}(U_{\be }\cap {\mathcal L})=\D \times
C_{\be}$.
\end{enumerate}
}
\end{definition}

A leaf $L$ of ${\cal L}$ is called a {\it limit leaf} if for some (or every) point $p\in L$, there exists a coordinate chart
$\varphi _{\beta }\colon \D \times (0,1)\to U_{\be }$ as in Definition~\ref{deflamination}
such that $p\in U_{\be }$ and $\varphi _{\be }^{-1}(p)=(x,t)$ with
$t$ belonging to the accumulation set of $C_{\be }$.

We will simply denote laminations by ${\mathcal L}$, omitting the
charts $\varphi _{\be }$ in ${\mathcal A}$.
%A lamination ${\mathcal L}$ is said to be a {\it foliation of $N$} if ${\mathcal L}=N$.
Every lamination ${\mathcal L}$ naturally decomposes into a
collection of disjoint, connected topological surfaces (locally given by $\varphi
_{\be }(\D \times \{ t\} )$, $t\in C_{\be }$, with the notation
above), called the {\it leaves} of ${\mathcal L}$. As usual, the
regularity of ${\mathcal L}$ requires the corresponding
regularity on the change of coordinate charts $\varphi _{\be }$.
A lamination ${\cal L}$ of $N$ is called a {\it foliation} of $N$ if ${\cal L}=N$.
%Note that if $\Delta \subset {\cal L}$ is any collection of leaves of ${\cal L}$, then
%the closure of the union of these leaves has the structure of a
%lamination within ${\cal L}$, which we will call a {\it sublamination.}
A lamination ${\cal L}$ of $N$ is said to be a {\it minimal lamination}
if all its leaves are smooth with zero mean curvature. Since the leaves of ${\cal L}$ are pairwise
disjoint, we can consider the second fundamental form
$|\sigma _{\cal L}|$ of ${\cal L}$, which is the function defined at any point $p\in {\cal L}$
as $|\sigma _L|(p)$, where $L$ is the unique leaf of ${\cal L}$ passing through $p$.

There are three key results that we will need from~\cite{mpr10}, and which we
list below for the readers convenience.

\begin{theorem} [Local Removable Singularity Theorem~\cite{mpr10}]
\label{tt2.2}
A minimal lamination $\lc$ of a punctured ball $B_N(p,r)-\{ p\} $
in a Riemannian three-manifold $N$ extends to a minimal lamination
of $B_N(p,r)$ if and only if there exists a positive constant $C$
such that $|\sigma _{{\cal L}}|\, d_N(p,\cdot )\leq C$ in some subball.
\end{theorem}

\begin{definition}
{\rm
In the sequel, we will denote by $R=\sqrt{x_1^2+x_2^2+x_3^2}$ the radial distance
to the origin $\vec{0}\in \R^3$. A surface $M\subset \R^3$
has {\em quadratic decay of curvature} if there exists $C>0$ such that
$|K_M|R^2 \leq {C}$ on $M$. Analogously, a minimal lamination ${\cal L}$ of $\R^3$
or of $\R^3-\{ \vec{0}\} $
has quadratic decay of curvature if $|K_{\cal L}|R^2$ is bounded on ${\cal L}$, where
$|K_{\cal L}|$ is the function that associates to each point $p\in {\cal L}$ the
absolute Gaussian curvature of the unique leaf of ${\cal L}$ passing through $p$.
}
\end{definition}

\begin{theorem}[Quadratic Curvature Decay Theorem~\cite{mpr10}]
\label{tt2.4}
A complete, embedded minimal surface in $\R^3$
with compact boundary (possibly empty) has quadratic decay of
curvature if and only if it has finite total curvature. In
particular, a complete, connected embedded minimal surface $M\subset
\R^3$ with compact boundary and quadratic decay of curvature is
properly embedded in $\R^3$.
\end{theorem}

\begin{proposition}[Corollary~6.3 in~\cite{mpr10}]
\label{corollamin}
Let ${\cal L}$ be a non-flat minimal lamination
of $\R^3-\{ \vec{0}\} $. If ${\cal L}$ has quadratic decay of
curvature, then ${\cal L}$ consists of a single leaf, which extends
across $\vec{0}$ to a properly embedded minimal surface with finite total curvature
in $\R^3$.
\end{proposition}

\section{Proof of the Local Picture Theorem on the Scale of Curvature.}
\label{seclt1}

The proof of Theorem~\ref{thm3introd} stated in the introduction
uses a blow-up technique, where the scaling factors are given by
the norm of the second fundamental form
of $M$ at points of large {\it almost-maximal
curvature,} a concept which we develop below. After the blowing-up
process, we will find a limit which is a complete minimal surface
with bounded Gaussian curvature in $\R^3$, conditions which are known to imply
properness for the limit~\cite{mr8}. This properness property will lead to the
conclusions of Theorem~\ref{thm3introd}.

{\it Proof of Theorem~\ref{thm3introd}.}\
Since $N$ is homogeneously regular,
after a fixed constant scaling of the metric
of $N$ we may assume that the injectivity radius of $N$
is greater than 1. The first step in the proof is
to obtain special points $p'_n\in M$, called
{\it blow-up points} or {\it points of almost-maximal curvature.}
First consider an arbitrary
sequence of points $q_n\in M$ such that
$|\sigma _M|(q_n)\geq n$, which exists since
$|\sigma _M|$ is unbounded. Let $p'_n\in
B_{M}(q_n,1)$ be a maximum of the function
$h_n=|\sigma _M|\, d_M\left( \cdot ,\partial B_M(q_n,1 )
\right) $, where
%$B_M(q_n,1)$ denotes the intrinsic metric ball in $M$ centered at
%$q_n$ with radius 1 and
$d_M$ stands for the intrinsic distance on $M$.
We define $\l '_n=|\sigma _M|(p'_n)$. Note that for each $n\in \N$,
\[
\l '_n\geq \l '_nd_M(p'_n,\partial B_M(q_n,1))
=h_n(p'_n)\geq h_n(q_n)=|\sigma _M|(q_n)\geq n.
\]

Fix $t>0$. Since $\l '_n\to \infty $ as
$n\to \infty $, the sequence
$\{ \l '_nB_N(p'_n,\frac{t}{\l '_n})\} _n$ converges
to the open ball $\B (t)$ of $\R^3$ with its usual
metric, where we have used geodesic coordinates in $N$
 centered at $p'_n$ and identified $p'_n$ with
$\vec{0}$. Similarly, we can consider
$\{ \l '_nB_M(p'_n,\frac{t}{\l '_n})\} _n$ to be a
sequence of embedded, constant mean curvature
surfaces with non-empty topological boundary, all passing through $\vec{0}$
with norm of their second fundamental forms $1$ at this point.
We claim that the sequence of second fundamental forms of
$\l'_nB_M(p'_n,\frac{t}{\l '_n})$ is uniformly bounded. To see this,
pick a point $z_n\in B_M(p'_n,\frac{t}{\l '_n})$.
Note that for $n$ large enough, $z_n$ lies in
$B_M(q_n,1)$. Then,
\begin{equation}
\label{eq:*}
\frac{|\sigma _M|(z_n)}{\l '_n}=
\frac{h_n(z_n)}{\l '_n d_M(z_n,\partial B_M(q_n,1))}\leq
\frac{d_M(p'_n,\partial B_M(q_n,1))}{d_M(z_n,\partial B_M(q_n,1))}.
\end{equation}
By the triangle inequality, $d_M(p'_n,\partial B_M(q_n,1)) \leq
\frac{t}{\l '_n}+d_M(z_n,\partial B_M(q_n,1))$, and so,
\[
\frac{d_M(p'_n,\partial B_M(q_n,1))}{d_M(z_n,\partial B_M(q_n,1))}
\leq 1+\frac{t}{\l '_nd_M(z_n,\partial B_M(q_n,1))}
\]
\begin{equation}
\label{eq:**} \leq 1+\frac{t}{\l '_n\left( d_M(p'_n,\partial
B_M(q_n,1)) -\frac{t}{\l '_n}\right) }\leq 1+\frac{t}{n-t},
\end{equation}
which tends to $1$ as $n\to \infty $.

It follows that after extracting a subsequence,
the surfaces $\l '_nB_M(p'_n,\frac{t}{\l '_n})$
converge (possibly with non-constant integer or infinite multiplicity) to an embedded minimal
surface $M_{\infty }(t)$ contained in $\B (t)$ with
bounded Gaussian curvature, that passes through
$\vec{0}$ and with norm of its second fundamental form $1$ at the
origin; note that the topological boundary $\partial M_{\infty}(t)$ of $M_{\infty}(t)$ need
not be either smooth or contained in $\esf^2(t)$.
Consider the surface $M_{\infty }(1)$
together with the surfaces $\l '_nB_M(p'_n,\frac{1}
{\l '_n})$ that converge to it (after passing to a
subsequence). Note that $M_{\infty }(1)$ is contained
in $M_{\infty }=\bigcup _{t\geq 1}M_{\infty }(t)$,
which is a complete, injectively immersed minimal surface in $\R^3$, with
$\vec{0}\in M_{\infty }$ and $|\sigma _{M_{\infty }}|(\vec{0})=1$.
By construction, $M_{\infty }$ has bounded Gaussian curvature,
so it is properly embedded in $\R^3$~\cite{mr8}.

The next result describes how the surfaces $\l '_nB_M(p'_n,\frac{t}
{\l '_n})$ that give rise to the limit $M_{\infty }(t)$ sit nicely in space with
respect to the surface $M_{\infty }(t+1)$.

\begin{lemma}
\label{lemma3.1}
Given $t>0$, there exist $k\in \N$ such that if $n\geq k$, then
$\l '_nB_M(p'_n,\frac{t}{\l '_n})$ is  contained in a small regular neighborhood of $M_{\infty }(t+1)$ in $\R^3$. Furthermore,
$\l '_nB_M(p'_n,\frac{t}{\l '_n})$ is a small normal
graph over its projection to $M_{\infty }(t+1)$.
\end{lemma}
\begin{proof}
Let $\pi \colon \wt{M}_{\infty }\to M_{\infty }$ be the universal cover of $M_{\infty }$.
Choose a point $\wt{x}\in \pi ^{-1}(\{ \vec{0}\} )$.
Since $M_{\infty }$ is not flat, then $\wt{M}_{\infty }$ is not stable~\cite{cp1,fs1,po1} and thus,
there exists $R>0$ such that the intrinsic ball $B_{\wt{M}_{\infty }}(\wt{x},R)$
centered at $\wt{x}$ with radius $R$ is unstable.

Choose $t>0$. Since the closure $\overline{M}_{\infty }(t)$ of $M_{\infty }(t)$ is compact
(because $M_{\infty }(t)\subset \B (t)$),
then there exists $\de _t>0$ such that $\overline{M}_{\infty }(t)$ admits a regular
neighborhood $U(t)\subset \R^3$ of radius $\de _t>0$;
in particular, $U(t)$ is diffeomorphic to $\overline{M}_{\infty }(t)\times (-\de _t,\de _t)$
and we have a related normal projection $\Pi _t\colon U(t)\to \overline{M}_{\infty }(t)$.
Let $\pi _t\colon \wt{M}_{\infty }(t)\to \overline{M}_{\infty }(t)$ be the universal cover of
$\overline{M}_{\infty }(t)$ and let
$\pi _t\times \mbox{Id}\colon \wt{U}(t)\equiv \wt{M}_{\infty }(t)\times
(-\de _t,\de _t)\to \overline{M}_{\infty }(t)\times (-\de _t,\de _t)$ be the universal cover of $U(t)$,
each one endowed with the pulled back metric.
Therefore, we also have a
normal projection $\wt{\Pi } _t\colon \wt{U}(t)\to \wt{M}_{\infty }(t)$
such that $\pi _t\circ \wt{\Pi }_{t}=\Pi _{t}\circ (\pi _t\times \mbox{Id})$.

Since the sequence $\{ \l '_nB_M(p'_n,\frac{t}{\l '_n})\} _n$ converges to $M_{\infty }(t)$ as $n\to \infty $,
there exists $n_0=n_0(t)\in \N$ such that for every $n\geq n_0$, we have
$\l '_nB_M(p'_n,\frac{t}{\l '_n})\subset U(t+1)$, which clearly implies
the first sentence in the statement
of Lemma~\ref{lemma3.1}. To prove the second statement, we argue by contradiction. Suppose
that for some $t>0$, there exists a sequence of integer numbers $n(m)\geq n_0(t)$ tending
to $\infty $ such that
for each $m$, $\l '_{n(m)}B_M(p'_{n(m)},\frac{t}{\l '_{n(m)}})$ fails to be a normal
graph over its projection to $M_{\infty }(t+1)$. To keep the notation simple, we will
relabel $n(m)$ as $n$. This means that for each $n\geq n_0(t)$, there exist two distinct points
$q_n(1),q_n(2)\in \l '_{n}B_M(p'_{n},\frac{t}{\l '_{n}})$ such that
$\Pi _{t+1}(q_n(1))=\Pi _{t+1}(q_n(2))$. As the sequence
$\{ \Pi _{t+1}(q_n(1))\} _{n\geq n_0(t)}$ lies in the compact set $\overline{M}_{\infty }(t+1)$, after
extracting a subsequence we may assume that the $\Pi _{t+1}(q_n(1))=\Pi _{t+1}(q_n(2))$
converge as $n\to \infty $ to a point $q_{\infty }\in \overline{M}_{\infty }(t)$. Therefore, there
exists some $\ve =\ve (t)>0$ small such that for each $n\geq n_0(t)$, $\l _n'B_M(p_n',\frac{t+1}{\l _n'})$
contains two disjoint disks  $D_1(n),D_2(n)$ in such a way that
\begin{equation}
\label{eq:smalldisks}
(\Pi _{t+1})|_{D_i(n)}\colon D_i(n)\to B_{M_{\infty }}(q_{\infty },\ve )\quad
\mbox{is a diffeomorphism, }i=1,2,
\end{equation}
where $B_{M_{\infty }}(q_{\infty },\ve )\subset M_{\infty }(t+1)$ denotes the geodesic disk in $M_{\infty }$ centered at $q_{\infty }$ with radius $\ve $. Consider the universal covering
\[
\pi =\pi _{t+R+2}\colon \wt{M}_{\infty }(t+R+2)\to \overline{M}_{\infty }(t+R+2).
\]

Choose a point $\wt{q}_{\infty }\in \pi ^{-1}(\{ q_{\infty }\} )\subset \wt{M}_{\infty }(t+R+2)$ such that the distance from $\wt{x}$ to $\wt{q}$ is less than or equal to $t$ (this can be done since $q_{\infty }\in
\overline{B}_{M_{\infty }}(\vec{0},t)$).
 We will find the desired contradiction
by constructing a positive Jacobi function on the closed intrinsic ball
$\overline{B}_{\wt{M}_{\infty }(t+R+2)}(\wt{q}_{\infty },t+R)$,
which is impossible since this last closed ball contains
the unstable domain $B_{\wt{M}_{\infty }}(\wt{x},R)$ by the triangle inequality.

Take $n_0=n_0(t)$ large such that for all $n\geq n_0$, $\l _n'\overline{B}_M(p_n',\frac{t+R+1}{\l '_n})$
lies in $U(t+R+2)$. Hence, we can lift $\l _n'\overline{B}_M(p_n',\frac{t+R+1}{\l '_n})$ to
$\widetilde{U}(t+R+2)$ via the covering $\pi \times \mbox{Id}$.
Note that
\[
V_n(t,R):=(\pi \times \mbox{Id})^{-1}\left[ \l _n'
B_M(p_n',{\textstyle \frac{t+R+1}{\l '_n}})\right]
\]
 is a possibly disconnected, non-compact surface in $\widetilde{U}(t+R+2)$.
As $B_{M_{\infty }}(q_{\infty },\ve )\subset M_{\infty }(t+1)$ is a disk
and $\pi $ is a Riemannian covering, then
$B_{M_{\infty }}(q_{\infty },\ve )$ lifts to the geodesic disk $B_{\wt{M}_{\infty }(t+R+2)}(\wt{q}_{\infty },\ve )$
in $\wt{M}_{\infty }(t+R+2)$ with center $\wt{q}_{\infty }$ and radius $\ve $. Using
(\ref{eq:smalldisks}), we can lift $D_1(n)$, $D_2(n)$ to small disks $\wt{D}_1(n),
\wt{D}_2(n)\subset V_n(t,R)$ such that
\begin{equation}
\label{eq:smalldisks1}
(\wt{\Pi }_{t+R+2})|_{\wt{D}_i(n)}\colon \wt{D}_i(n)\to B_{\wt{M}_{\infty }(t+R+2)}(\wt{q}_{\infty },\ve )\quad
\mbox{is a diffeomorphism, }i=1,2.
\end{equation}

As the closed intrinsic metric ball $\overline{B}_{\wt{M}_{\infty }(t+R+2)}(\wt{q}_{\infty },R)$ is compact
and the sequence of geodesic disks $\{ \l _n'B_M(p_n',\frac{t+R+1}{\l _n'})\} _n$ converges smoothly
to $M_{\infty }(t+R+1)\subset M_{\infty }(t+R+2)$, then given $\mu >0$ small there exists $n_1=n_1(t,\mu )
\in \N $ (we may assume $n_1\geq n_0(t)$) such that for every $n\geq n_1$, the normal projection
$\wt{\Pi }_{t+R+2}$ restricts
to $V_n(t,R)$ as a {\it $\mu $-quasi-isometry,} in the sense that the ratio between the length of
tangent vectors at points in $V_n(t,R)$ and their images through the differential of $\wt{\Pi }_{t+R+2}$
lies in the range $[1-\mu ,1+\mu ]$. In particular, every radial geodesic arc $\g \subset
\overline{B}_{\wt{M}_{\infty }(t+R+2)}(\wt{q}_{\infty },t+R)$ starting at $\wt{q}_{\infty }$ and ending at
$\partial \overline{B}_{\wt{M}_{\infty }(t+R+2)}(\wt{q}_{\infty },t+R)$ can be uniquely lifted to a pair of
disjoint embedded arcs $\g _1(n),\g _n(n)\subset V_n(t,R)$ starting at the points $\wh{q}_i(\infty )
:=\left( \wt{\Pi }_{t+R+2}|_{\wt{D}_i(n)}\right) ^{-1}(\wt{q}_{\infty })\in \wt{D}_i(n)$, and these arcs have length
less than or equal to $t+R+1$, for $i=1,2$. When $\g $ varies in the set of radial geodesic arcs
starting at $\wt{q}_{\infty }$ and ending at $\partial \overline{B}_{\wt{M}_{\infty }(t+R+2)}(\wt{q}_{\infty },t+R)$,
the union of the related lifted arcs $\g _1(n)$, $\g _2(n)$ give rise to closed, disjoint topological disks
$\overline{D}_{1,R}(n),\overline{D}_{2,R}(n)\subset V_n(t,R)$, with the property that
\begin{equation}
\label{eq:quasiisom}
\wt{\Pi }_{t+R+2}|_{\overline{D}_{i,R}(n)}\colon \overline{D}_{i,R}(n)\to
\overline{B}_{\wt{M}_{\infty }(t+R+2)}(\wt{q}_{\infty },t+R)\quad
\mbox{is a $\mu $-quasi-isometry, }i=1,2.
\end{equation}

Property (\ref{eq:quasiisom}) implies that $\overline{D}_{i,R}(n)$ can be expressed as the graph
of a function
\[
u_{i,n}\colon \overline{B}_{\wt{M}_{\infty }(t+R+2)}(\wt{q}_{\infty },t+R)\to \R ,\quad i=1,2.
\]
As the sequence of disks $\{ \overline{D}_{i,R}(n)\} _n$ converges as $n\to \infty $ to
$\overline{B}_{\wt{M}_{\infty }(t+R+2)}(\wt{q}_{\infty },t+R)$ for $i=1,2$, then a
subsequence of the functions $\frac{u_{1,n}-u_{2,n}}{(u_{1,n}-u_{2,n})(\wt{q}_{\infty })}$
converges as $n\to \infty $ a non-zero Jacobi function
on $\overline{B}_{\wt{M}_{\infty }(t+R+2)}(\wt{q}_{\infty },t+R)$, which has non-zero sign since
$\overline{D}_{1,R}(n)$, $\overline{D}_{2,R}(n)$ are disjoint. This contradicts the unstability
of $\overline{B}_{\wt{M}_{\infty }(t+R+2)}(\wt{q}_{\infty },t+R)$ and finishes the proof.
\end{proof}

\begin{lemma}
\label{lemma3.2}
For all $R>0$, there exist $t>0$ and
$k\in \N$ such that if $m\geq k$, then the component of
$\left[ \l '_mB_M(p'_m,\frac{t}{\l '_m})\right]
\cap \overline{\B }(R)$ that passes through $\vec{0}$ is compact
has its boundary on $\esf ^2(R)$.
\end{lemma}
\begin{proof}
We fix $R>0$ and suppose that $M_{\infty }$ intersects transversely the sphere $\esf ^2(R)$
(this transversality property holds for almost every $R$; we will prove the lemma assuming
this transversality property, and the lemma will hold for every $R$ after a continuity
argument). As $M_{\infty }$ intersects transversely $\esf ^2(R)$, there exists
$\ve =\ve (R)>0$ small such that $M_{\infty }$ intersects transversely $\esf ^2(R')$
for all $R'\in [R,R+\ve ]$. Given $R' \in [R,R+\ve ]$, let $\Delta (R')$ be the
component of $M_{\infty }\cap \overline{\B }(R')$ that contains $\vec{0}$. Note that $\Delta (R+\ve )$ is
compact and contained in the interior of some intrinsic geodesic disk $B_{M_{\infty }}(\vec{0},R_1)$,
$R_1=R_1(R)\geq R+\ve $. Also note that $B_{M_{\infty }}(\vec{0},R_1)$ is the limit as $n\to \infty $ of the
intrinsic geodesic disks $\l _n'B_{M}(p_n',\frac{R_1}{\l _n'})$.

Applying Lemma~\ref{lemma3.1} to $t=R_1+1$, we obtain an integer number $k=k(R)$ such that for each $m\geq k$,
the surface $\l '_mB_M(p'_m,\frac{R_1+1}{\l '_m})$ is  contained in a small regular neighborhood $U(R_1+2)$
of radius $\de (R_1+2)>0$ of $M_{\infty }(R_1+2)$ and $\l '_mB_M(p'_m,\frac{R_1+1}{\l '_n})$ is a small normal
graph over its projection $W_m(R_1+1)$ to $M_{\infty }(R_1+2)$, i.e.,
\[
(\Pi _{R_1+2})|_{\l '_mB_M\left( p'_m,{\textstyle \frac{R_1+1}{\l '_m}}\right)}\colon
\l '_mB_M(p'_m,{\textstyle \frac{R_1+1}{\l '_m}})\to W_m(R_1+1)
\quad \mbox{is a diffeomorphism, }\forall m\geq k.
\]
Here we are using the same notation as in the proof of Lemma~\ref{lemma3.1}; therefore,
$\Pi _{R_1+2}\colon U(R_1+2)\to M_{\infty }(R_1+2)$ is the normal projection associated to
the regular neighborhood $U(R_1+2)$
% of radius $\de (R_1+2)>0$
of $M_{\infty }(R_1+2)$.

Note that as $\Delta (R+\ve )
\subset B_{M_{\infty }}(\vec{0},R_1)$, then there exists $k_1\geq k$
such that $\Delta (R+\ve )\subset
W_m(R_1+1)$, $\forall m\geq k_1$. As the distance between $\partial B_{M_{\infty }}(\vec{0},R_1)$
and $\partial B_{M_{\infty }}(\vec{0},R_1+1)$ is 1, and $W_m(R_1+1)$ converges to
$B_{M_{\infty }}(\vec{0},R_1+1)$ as $m\to \infty $, then we assume that for $m$ sufficiently large,
$B_{M_{\infty }}(\vec{0},R_1)$ lies in $W_m(R_1+1)$.
Since $\Delta (R+\ve )\subset  B_{M_{\infty }}(\vec{0},R_1)$, then we conclude
that $\Delta (R+\ve )\subset W_m(R_1+1)$ for all $m$ sufficiently large.
This implies that the compact domains
\[
\Omega (m):=\left( (\Pi _{R_1+2})|_{\l _m'B_M\left( p'_m,\frac{R_1+1}{\l _m'}\right) }\right) ^{-1}
\left[ \Delta (R+\ve )\right]
\subset \l _m'B_M(p'_m,{\textstyle \frac{R_1+1}{\l _m'}})
\]
converge smoothly to $\Delta (R+\ve )$ as $m\to \infty $. In particular, the sequence of boundaries
$\partial \Omega (m)$
converge to the boundary of $\Delta (R+\ve )$, which is contained in the sphere $\esf^2(R+\ve )$.
This implies that for $m$ large, $\partial \Omega (m)$ lies outside $\overline{\B }(R)$.
Now the lemma follows by taking $t=R_1+1$, as the component of
$\left[ \l '_mB_M(p'_m,\frac{t}{\l '_m})\right]
\cap \overline{\B }(R)$ passing through $\vec{0}$ coincides with the component of $\Omega (m)\cap \overline{\B }(R)$ passing through $\vec{0}$.
\end{proof}

\begin{remark}
  {\rm
Note that Lemmas~\ref{lemma3.1} and~\ref{lemma3.2} remain valid under the
hypotheses that $M_{\infty }$ is not flat and the Gaussian curvature of
$\{ \l _nB_M(p_n,\frac{t}{\l _n})\} _n$ is uniformly bounded for each $t>0$
(with the bound depending on $t$).
}
\end{remark}

We next finish the proof of Theorem~\ref{thm3introd}.
Applying Lemma~\ref{lemma3.2} to $R_n=\l '_n$, we obtain $t(n)>0$
and $k(n)\in \N $ such that if $M_n$
denotes the component of $B_M\left( p'_{k(n)},\frac{t(n)}{\l '_{k(n)}}\right)
\cap  \overline{B}_N\left( p'_{k(n)},\frac{\l '_n}{\l '_{k(n)}}\right) $
that contains $p'_{k(n)}$, then $M_n$ is compact and
its boundary satisfies $\partial M_n\subset \partial B_N\left(
 p'_{k(n)},\frac{\l '_n}{\l '_{k(n)}}\right)$.
Clearly this compactness property remains valid if we
increase the value of $k(n)$. Hence, we may assume
without loss of generality that
\[
t(n)(n+1)<k(n)\ \mbox{ for all }n, \qquad
\frac{\l '_n}{\l '_{k(n)}}\to 0\ \mbox{ as }
n\to \infty .
\]

We now define $p_n=p'_{k(n)}$, $\ve _n= \frac{\l '_n}{\l '_{k(n)}}$
and $\l _n=\l '_{k(n)}$. Then it is easy
to check that the $p_n,\ve _n,\l _n$ and $M_n$ satisfy
the conclusions stated in Theorem~\ref{thm3introd} (in
order to prove item~{2} in the statement
of Theorem~\ref{thm3introd}, simply note that equations
(\ref{eq:*}) and (\ref{eq:**}) imply that
$\frac{|\sigma _{M_n}|}{\l _n}=
\frac{|\sigma _{M_n}|}{\l '_{k(n)}}\leq
   1+\frac{t(n)}{k(n)-t(n)}<1+\frac{1}{n}$, where the last inequality
follows from $t(n)(n+1)<k(n)$). This finishes the proof of
Theorem~\ref{thm3introd}.
{\hfill\penalty10000\raisebox{-.09em}{$\Box$}\par\medskip}
\par
\vspace{.2cm}

\begin{remark}
{\rm
If the surface $M\subset N$ in Theorem~\ref{thm3introd}
were properly embedded, then the arguments needed to
carry out its proof could be formulated in a more
standard manner by using the techniques developed in
the papers~\cite{mpr3,mr8}. It is the non-properness of
$M$ that necessitates being more careful
in defining the limit surface $M_{\infty}$ and in proving additional
properties of how it arises as a limit surface of compact embedded minimal surfaces
that appear in the blow-up procedure in the proof of Theorem~\ref{thm3introd}.
}
\end{remark}

\section{Applications of Theorem~\ref{thm3introd}.}

In any flat three-torus $\Te ^3$, there exists a
sequence $\{ M_n\} _n$ of embedded, compact minimal surfaces of genus three,
such that the areas of these surfaces diverge to infinity~\cite{me6}
(a similar result holds
for any genus $g\geq 3$, see Traizet~\cite{tra5}). After choosing a
subsequence, these surfaces converge to a minimal foliation of
$\Te^3$ and the convergence is smooth away from two points. Since by
the Gauss-Bonnet formula, these surfaces have absolute total
curvature $8\pi $, this example demonstrates a special case of Theorem~\ref{corol5.4}
below. Before stating this result, we
recall a somewhat standard result concerning limits of minimal surfaces. A similar
statement can be found in item~5 of Lemma A.1 in Meeks and Rosenberg~\cite{mr13}.
\begin{lemma}
\label{lemma4.2}
  Suppose that $\{ M_n\} _n$ is a sequence of complete embedded minimal surfaces
  in a Riemannian three-manifold $N$, which converge to minimal lamination ${\cal L}$
of $N$. Let $L$ be a leaf of ${\cal L}$ which is either a limit leaf of ${\cal L}$
or it is an isolated leaf and in this case, the convergence of the sequence $\{ M_n\} _n$
to $L$ has multiplicity greater than 1. Then, the two-sided cover of $L$ is stable.
\end{lemma}
\begin{proof}
  If $L$ is a limit leaf of ${\cal L}$, then the main theorem in~\cite{mpr18} insures
that the two-sided cover of $L$ is stable. Next suppose that $L$ is an isolated leaf of
${\cal L}$ and that the convergence of the $M_n$ to $L$ has multiplicity greater than 1.
Consider a compact subdomain $D\subset L$ and let $D_{\de }$ be a regular neighborhood
of $D$ in $N$ of small radius $\de >0$. After possibly lifting to a two-sheeted cover of $D_{\de }$,
we may assume that $D$ is two-sided.
Thus $D_{\de }$ is diffeomorphic to $D\times [-\de ,\de ]$.
Since $L$ is isolated as a leaf of ${\cal L}$, then the `top' and `bottom' sides
$D\times \{ -\de \}$ and $D\times \{ \de \} $ of $D_{\de }$ can be assumed to be disjoint from ${\cal L}$
and, since they are compact, they are also disjoint from the surfaces $M_n$ for $n$
sufficiently large. Another consequence of the convergence of the $M_n$ to $L$ and
of the compactness of $D$ is that we may assume that the $M_n\cap D_{\de }$ are locally graphs
over their projections to $D$. Consider the sequence of minimal laminations $\{ \overline{M_n\cap D_{\de }}\} _n$,
which converges to $D$. Note that for $n$ large, each normal unit speed geodesic $\g _x$ in $D_{\de }$ starting at
a point $x\in D$ intersects the lamination $\overline{M_n\cap D_{\de }}$ in a closed set which has a highest
point $\g _x(t_n(x))$ and a lowest point $\g _x(s_n(x))$, for some real numbers $s_n(x)\leq t_n(x)$.
As the  multiplicity of the limit $M_n\to L$
is greater than one, then $s_n(x)<t_n(x)$ for each $n\in \N $.
Consider the function $u_n(x)=t_n(x)-s_n(x)$, for all $x\in D$.
Since the lamination $\overline{M_n\cap D_{\de }}$ is minimal for each $n\in \N$,
then after normalizing $u_n$ to be 1 at some point $p\in \mbox{Int}(D)$, a standard argument shows that
these normalized functions converge to a positive Jacobi function of $D$, which implies
that $D$ is stable. Finally, $L$ is stable as every compact subdomain $D\subset L$ is stable.
\end{proof}

\begin{theorem}
\label{corol5.4} Suppose $\{ M_n\} _n$ is a sequence of complete,
embedded minimal surfaces in a Riemannian three-manifold $N$, such
that  there exists a open covering of $N$ and $\int _{M_n\cap B}|\sigma _n|^2$
is uniformly bounded for any open set $B$ in this
covering (here $\sigma _n$ denotes the second fundamental form of $M_n$).
Then, there exists a subsequence of $\{ M_n\} _n$ that converges to
a minimal lamination ${\cal L}$ of $N$, and the singular
set of convergence of the $M_n$ to ${\cal L}$, defined as
\begin{equation}
\label{singsetconv}
\hspace{-.35cm}S({\cal L})=\left\{ \rule{0cm}{.4cm}p\in {\cal L}\ | \ \mbox{the sequence $\{ |\sigma _{M_n}| \} _n$
is not uniformly bounded in any neighborhood of $p$}\right\} ,
\end{equation}
is closed and discrete. Furthermore:
\begin{enumerate}
\item If $L$ is a
limit leaf of $\lc$ or a leaf with infinite multiplicity as a limit of the surfaces $M_n$,
then the two-sided cover of $L$ is stable and $L$ is totally geodesic.
\item If each $M_n$ is connected and
$N$ is compact, then ${\cal L}$ is compact and connected in the
subspace topology.
\end{enumerate}
\end{theorem}
\begin{proof}
We will distinguish between {\it good and bad points} of $N$, depending on whether
or not the surfaces $M_n$ have a good behavior around the point to take limits; the
set $A\subset N$ of bad points will be then proven to be discrete and closed in $N$, and
we will produce a limit minimal lamination ${\cal L}$ of the $\{ M_n\} _n$ in
$N-A$. The final step in the proof of the first statement in the theorem
will be to show that ${\cal L}$ extends as a minimal lamination across $A$.

Let $q$ be a point in $N$. We will say that $q$ is a {\it bad} point
for the sequence $\{ M_n\} _n$ if there exists a subsequence $\{
M_{n_k}\} _{k}\subset \{ M_n\} _n$ such that
\[
\int _{M_{n_k}\cap B_N(q,\frac{1}{k})}|\sigma _{M_{n_k}}|^2\geq 2\pi ,\quad \mbox{ for all
$k\in \N $.}
\]
First note that we can replace the covering in the statement by a
countable open covering of $N$ by balls $B _i$, $i\in \N $. Assume
for the moment that $B _1$ contains a bad point $q_1$ for $\{ M_n\}
_n$. We claim that $B_1$ has a finite number of bad points after
replacing $\{ M_n\} _n$ by a subsequence. To see this, since $q_1$
is a bad point for $\{ M_n\} _n$, there exists a subsequence $\{
M'_k=M_{n_k}\} _{k}\subset \{ M_n\} _n$ such that the total
curvature of every $M'_k$ in $B_N(q_1,\frac{1}{k})$ is at least
$2\pi $. Suppose that $q_2\in B _1$ is another  bad point for $\{ M'_k\}
_k$. Then we find a subsequence $\{ M''_j=M_{k_j}\} _{j}\subset \{
M'_k\} _k$ such that the total curvature of every $M''_j$ in
$B_N(q_2,\frac{1}{j})$ is at least $2\pi $. In particular, for $j$
large, there are disjoint neighborhoods of $q_1$ and $q_2$ in $B_1$,
each with total curvature of $M''_j$
at least $2\pi $. Since $\{ \int _{M_n\cap B_1}|\sigma _n|^2\} _n$ is uniformly
bounded, this process of finding bad points and subsequences in
$B _1$ stops after a finite number of steps, which proves our
claim. A standard diagonal argument then shows that after replacing
the $M_n$ by a subsequence, the set of bad points $A\subset N$ for
$\{ M_n\} _n$ is a discrete closed set in $N$.

Suppose that $q\in N-A$. We claim that $\{ M_n\} _n$ has pointwise bounded
second fundamental form in some neighborhood of $q$.
Arguing by contradiction, suppose there exist points $p_n\in M_n$
converging to $q$ and such that $|\sigma _{M_n}|(p_n)\to \infty $ as $n\to \infty $.
Let $\ve _q=\frac{1}{2}d_N(q,A)>0$. By the Local Picture Theorem on the Scale
of Curvature, there exists a blow-up limit of the $M_n$ that converges to a
{\it non-flat,} properly embedded minimal surface in $\R^3$; as  the
total curvature of a non-flat, complete minimal surface in $\R^3$ is at least $4\pi $,
and the $L^2$-norm of the second fundamental form is invariant under rescalings
of the ambient metric, then we may assume that for $n$ large,
\[
\int _{M_n\cap B_N(q,r_n)}|\sigma _n|^2>2\pi ,
\]
where $r_n\searrow 0$ satisfies $d_N(q,p_n)<r_n<\frac{\ve _q}{2}$.
This clearly contradicts that $q\in N-A$, and so, we conclude that
$\{ M_n\} _n$ has pointwise bounded
second fundamental form in some neighborhood $U_q$ of $q$.
Therefore, there exists a minimal lamination
${\cal L}_q$ of $U_q$ such that a subsequence of the $M_n$ converges as $n\to \infty $
to ${\cal L}_q$ in $U_q$.
Another standard diagonal argument proves that after extracting a
subsequence, the $M_n$ converge to a minimal lamination ${\cal L}$
of $N-A$.

Next we show that ${\cal L}$ extends across ${\cal A}$ to a minimal lamination
of $N$. Consider the second fundamental form $|\sigma _{\cal L}|$ of ${\cal L}$.
We claim that $|\sigma _{\cal L}|$ does
not grow faster than linearly at any point $q\in A$ in terms of
the inverse of the extrinsic distance function to $q$: otherwise,
there exists a sequence of blow-up points
$p_n\in {\cal L}$
converging to a point $q\in A$ with $|\sigma _{L_n}|(p_n)d_N(p_n,q)$
unbounded, where $L_n$ is the leaf of ${\cal L}$ passing through
$p_n$. Using again the Local Picture Theorem on the Scale of
Curvature, we deduce that there exist disjoint small neighborhoods
$V(p_n)$ of $p_n$ in $L_n$, such that
\[
\int _{V(p_n)}|\sigma _{L_n}|^2>2\pi ,\quad \mbox{ for all $n\in \N$.}
\]
Since $M_n$ converges to ${\cal
L}$, this contradicts the hypothesis that
$\int _{M_n\cap B}|\sigma _n|^2$ is uniformly bounded for the open
set $B$ in the covering which contains $q$.
Once we know that $|\sigma _{\cal L}|$
does not grow faster than linearly at any point of the discrete closed set
$A$, Theorem~\ref{tt2.2} implies that
${\cal L}$ extends across $A$ to a minimal lamination of $N$.
Observe that by construction, the singular set of convergence
$S({\cal L})$ defined in~(\ref{singsetconv}) coincides with the set
$A$ of bad points. This proves the first sentence of the theorem.

We next prove item~1. Let $L$ be a
limit leaf of ${\cal L}$. By the main result in~\cite{mpr18},
the two-sided cover of $L$ is stable. If $L$ is not totally geodesic,
then there exists $q\in L-S({\cal L})$ such that $|\sigma _L|(q)
=4\ve >0$ (recall that $S({\cal L})$ is closed and discrete).
Then, there exists some open set $U\subset [L-S({\cal L})]\cap B$
such that $|\sigma _L|\geq 2\ve $ in $U$, where $B$ is an
open set in the covering that appears in the statement of the theorem, such
that $q\in B$. As $L$ is a limit leaf of ${\cal L}$ and $L$ is
the limit in $N-S({\cal L})$ of the $M_n$, then
there exist pairwise disjoint domains $U_n\subset M_n\cap B$ such that
$|\sigma _n|\geq \ve $ in $U_n$ for all $n$. This clearly contradicts
that $\int _{M_n\cap B}|\sigma _n|$ is bounded. This proves item~1 of the
theorem when $L$ is a limit leaf of {\cal L}.

If $L$ is not a limit
leaf of ${\cal L}$ but the multiplicity of the limit $\{ M_n\} _n\to L$
is infinity, then Lemma~\ref{lemma4.2} insures that the two-sided cover of
$L-S({\cal L})$ is stable. Given a compact subdomain $D$ of the two-sided cover
$\wh{L}$ of
$L$, the fact that $S({\cal L})$ is closed and discrete implies that $D\cap S({\cal L})$
is finite. After enlarging slightly $D$, we can assume that
$D\cap S({\cal L})$ lies in the interior of $D$. As $\wh{L}-S({\cal L})$ is stable,
then $D-S({\cal L})$ is also stable. A standard argument in elliptic theory
then shows that $D$ is also stable, and thus $\wh{L}$ is also stable, as desired.
The arguments in the last paragraph to show that $L$ is totally geodesic can be adapted easily
to this case, since the multiplicity of the limit $\{ M_n\} _n\to L$
is infinity. Now item~1 of the theorem is proved.

The proof of item~2 of the theorem is straightforward and we leave it for the reader.
\end{proof}

\begin{remark}
{\rm
Under the hypotheses and notation of Theorem~\ref{corol5.4},
we cannot conclude that ${\cal L}$ is path-connected
in the subspace topology when $M$ is connected and $N$
is compact: % (see the last statement of Theorem~\ref{corol5.4}):
 a counterexample can be
found for geodesic laminations in Example~3.5 of~\cite{mpr19},
and this example can be adapted to produce a minimal lamination
counterexample simply by taking products with $\esf^1$.
}
\end{remark}

We now give another application of Theorem~\ref{thm3introd}, which is
a partial positive answer to Conjecture~1.7 in~\cite{mpr13}.
Given a two-sided minimal surface
$M$ in a flat three-manifold $N$ and given $a>0$, we say that $M$ is
{\it $a$-stable} if for any compactly supported smooth function
$u\in C^\infty_0(M)$, we have
\begin{equation}\label{eq:a-stable1}
\int_{M}(|\nabla u|^2 + aKu^2) \geq 0,
\end{equation}
where $\nabla u$ stands for the gradient of $u$ and $K$ is the
Gaussian curvature of $M$ (the usual stability condition for the area
functional corresponds to the case $a=2$). More generally, we say
that $M$ has finite index of $a$-stability if there is a bound on
the number of negative eigenvalues (counted with multiplicities)
of the operator $\Delta -aK$ acting on smooth functions defined on
compact subdomains of $M$.

The authors conjectured in~\cite{mpr13} (Conjecture~1.7) that if $a>0$ and $M$ is a
two-sided, complete, embedded, $a$-stable minimal surface in a complete flat three-manifold
$N$, then $M$ is totally geodesic (flat). Do Carmo and Peng~\cite{cp1},
Fischer-Colbrie and Schoen~\cite{fs1} and Pogorelov~\cite{po1} proved independently
that if $M$ is a  complete, orientable $a$-stable minimal surface immersed in $\R^3$,
for $a\geq 1$, then $M$ is a plane. This result was later improved by Kawai~\cite{kaw1} to
$a> 1/4$, see also Ros~\cite{ros2}. In~\cite{mpr13}, the authors proved the conjecture
(for every $a>0$ and in any complete flat three-manifold $N$) under
 the additional hypotheses that $M$ is embedded and it has finite genus.
On the other hand, for small values of $a>0$, there exist
complete, non-flat, immersed, $a$-stable minimal surfaces in $\R^3$:
for instance, apply Lemma~6.3 in~\cite{mpr13} to the universal cover
of any doubly periodic Scherk minimal surface. In the next result we will
prove the conjecture for $a>1/8$ assuming solely that $M$ is embedded.
It should be also mentioned that Fischer-Colbrie~\cite{fi1} proved item~{3}
of the next result in the case $a\geq 1$, independently of whether or not $M$ is
embedded.

\begin{theorem}
\label{corol2.5}
Let $a\in (1/8,\infty )$ and $M$ be a two-sided, complete, embedded,  minimal surface with
compact boundary $\partial M$ in a complete, flat three-manifold $N$. Then:
\begin{enumerate}
\item If $\partial M=\mbox{\O }$ and $M$ is $a$-stable,
then $M$ is totally geodesic.
\item There exists $C>0$ (independent of $M,N$) such that if $\partial M\neq \mbox{\O }$ and
$M$ is $a$-stable, then $|\sigma _M|\, d_M(\cdot ,\partial M)\leq C$ and $M$ has finite total curvature.
\item If $N=\R^3$ and $M$ has finite index of
$a$-stability, then $M$ has finite total curvature.
\end{enumerate}
\end{theorem}
\begin{proof}
Suppose that $M\subset N$ has empty boundary and is $a$-stable. After lifting
$M$ to $\R^3$ and applying Theorem~\ref{thm3introd} (note that
$a$-stability is preserved after lifting to a covering, rescaling and taking smooth limits),
we can assume that $M$ has bounded Gaussian curvature and $N=\R^3$, in particular $M$ is
proper~\cite{mr8}. A straightforward application of the maximum
principle at infinity implies that $M$ has an embedded regular
neighborhood of fixed positive radius and so, $M$ has at most
cubical extrinsic area growth, see Meeks and Rosenberg~\cite{mr7}.
The following applications of previous results show that
$M$ is homeomorphic to $\C $ or to $\C -\{ 0\} $:
\begin{itemize}
\item If $a>1/4$, then apply Theorem~2.9 in~\cite{mpr19} (see also Castillon~\cite{cas1}).
\item If $a=1/4$, then apply part~(ii) of Theorem~1.2 in
Berard and Castillon~\cite{beca1}.
\item If $\frac{1}{8}<a<\frac{1}{4}$, then the cubical extrinsic area growth property of
$M$ implies that the intrinsic area growth of
$M$ is $k_a$-subpolynomial, where $k_a=2+\frac{4a}{1-4a}$. This means that
the limit as $r\to \infty $ of the area of the intrinsic ball in $M$ of radius $r$
(centered at any fixed point) divided by $r^{k_a}$ is zero. Now apply part~(iii)
of Theorem~1.2 in~\cite{beca1}.
\end{itemize}
As $M$ is a properly embedded minimal surface in $\R^3$ homeomorphic to $\C $ or to $\C -\{ 0\} $,
then $M$ is either a plane, a catenoid~\cite{sc1} or a helicoid~\cite{mr8}.
As both the catenoid and the helicoid are $a$-unstable for all $a>0$ by Proposition~1.5
in~\cite{mpr13}, we deduce item~{1} of the theorem.

Item { 2} follows from item~{1} by a rescaling argument on the scale of
curvature that is given in the proof of Theorem~\ref{thm3introd} (also see
the proof of Theorem~15 in~\cite{mpe1} for a similar argument).

To see item~{3,} suppose that $M\subset \R^3$ has finite
index of $a$-stability. By Proposition~1 in Fischer-Colbrie~\cite{fi1},
there exists a compact domain $\Omega \subset M$
such that $M-\mbox{Int}(\Omega )$ is $a$-stable. In this situation,
the already proven item~{2} of this theorem implies $M$ has quadratic
decay of curvature. Then, Theorem~\ref{tt2.4} implies that
$M$ has finite total curvature.
%
%
%intrinsic quadratic curvature decay for $M-\mbox{Int}(\Omega )$, i.e.,
%$|K_M|d_M(\cdot, \partial \Omega )^2$ is bounded in $M-\mbox{Int}(\Omega )$.
%Now consider the function $R=\sqrt{x_1^2+x_2^2+x_3^2}$ in $\R^3$.
%As $\partial \Omega $ is compact, then there exists $R_0>0$ such that $\partial
%\Omega \subset \overline{\B }(R_0)$. As intrinsic distance dominates
%extrinsic distance, then we have in $M-\B (R_0)$ that
%\begin{equation}
%\label{eq:proofthm4.4}
%\frac{R}{d_M(\cdot ,\partial \Omega )}\leq \frac{R}{d_M(\cdot ,M\cap \esf ^2(R_0))}\leq
%\frac{R}{R-R_0}.
%\end{equation}
%Since the right-hand-side of (\ref{eq:proofthm4.4}) is bounded in $M-\B (R')$ for any $R'>R_0$
%and $M$ is proper, we conclude that the function $|K_M|R^2$
%is bounded on $M-\mbox{Int}(\Omega )$. In this situation, Theorem~\ref{tt2.4} implies that
%$M-\mbox{Int}(\Omega )$ has finite total curvature.
\end{proof}

\section{Proof of the Dynamics Theorem.}
\label{sec8}

Our next goal is to prove the Dynamics Theorem (Theorem~\ref{thm4introd})
stated in the introduction.
Regarding the notions introduced in Definition~\ref{def1.8}, we make the following
observations.
\begin{description}
\item[(i)] If $\Sigma \in D(M)$, then $D(\Sigma )
\subset D(M)$ (this follows by considering double limits).
\end{description}
This property allows us to consider $D$ as an operator $D\colon D(M)\to {\cal P}(D(M))$,
where ${\cal P}(D(M))$ denotes the power set of $D(M)$.
\begin{description}
\item[(ii)] If $\Sigma \in D(M)$ and $D(\Sigma )=\mbox{\O }$, then
$\{ \Sigma \} $ is a minimal $D$-invariant set.

\item[(iii)] $\Sigma \in D(M)$ is quasi-dilation-periodic if and only if
$\Sigma \in D(\Sigma )$.

\item[(iv)] Any minimal element $\Sigma \in D(M)$ is contained in
a unique minimal $D$-invariant set.

\item[(v)] If $\Delta \subset D(M)$ is a minimal $D$-invariant set
and $\Sigma \in \Delta $ satisfies $D(\Sigma )\neq \mbox{\O }$, then
$D(\Sigma )=\Delta $ (otherwise $D(\Sigma )$ would be a proper
non-empty $D$-invariant subset of $\Delta $). In particular,
$\Sigma$ is quasi-dilation-periodic.

\item[(vi)] If $\Delta \subset D(M)$ is a $D$-invariant set and
$\Sigma \in \Delta $ is a minimal element, then the unique minimal
$D$-invariant subset $\Delta'$ of $D(M)$ which contains $\Sigma $
satisfies $\Delta'\subset \Delta $ (otherwise $\Delta'\cap \Delta $
would be a proper non-empty $D$-invariant subset of $\Delta'$).
\end{description}

{\it Proof of Theorem~\ref{thm4introd}.}\
First assume that $M$ has finite total curvature. Then, its total
curvature outside of some ball in space is less than $2\pi $, and
so, any $\Sigma \in D(M)$ must have total curvature less than $2\pi
$, which implies $\Sigma $ is flat. This gives that
 $D(M)=\mbox{\O }$.

Reciprocally, assume that $D(M)=\mbox{\O }$ and $M$
does not have finite total curvature. By
Theorem~\ref{tt2.4}, $M$ does not have quadratic
 decay of curvature, and so, there exists a divergent
sequence of points $z_n\in M$ with
   $|K_M|(z_n)|z_n|^2\to \infty $.
Let $p_n\in \B (z_n,\frac{|z_n|}{2})$ be a maximum of
the function $h_n=|K_M|d_{\R ^3}
(\cdot ,\partial \B (z_n,\frac{|z_n|}{2}))^2$.
 Note that $\{ p_n\} _n$ diverges in
$\R^3$ (because $|p_n|\geq \frac{|z_n|}{2}$). We define
$\l_n=\sqrt{|K_M|(p_n)}$. By similar arguments as those in the proof of
Theorem~\ref{thm3introd} applied to $M\cap \B \left( p_n,\frac{\sqrt{h_n(p_n)}}{2\l _n}
\right) $,
the sequence $\{ \l _n(M-p_n)\} _n$
converges (after passing to a subsequence)  to a minimal lamination
${\cal L}$ of $\R^3$ with a non-flat leaf $L$ which passes through
$\vec{0}$ and satisfies $|K_L|(\vec{0})=1$. Furthermore, the curvature
function $K_{\cal L}$ of ${\cal L}$ satisfies $|K_{\cal L}|\leq 1$
and so, the  leaf $L$ of ${\cal L}$ passing through $\vec{0}$ is
properly embedded in $\R^3$, as are all the leaves
of ${\cal L}$; see~\cite{mr7} for this properness result. By the Strong Half-space Theorem,
${\cal L}$ consists just of $L $, and the convergence of the
surfaces $\l _n(M-p_n)$ to $L$ has multiplicity one (by Lemma~\ref{lemma4.2}
since the two-sided cover of $L$ is not stable as $L$ is not flat). Therefore, $L\in D_1(M)$, which contradicts that
$D(M)=\mbox{\O }$. This proves the equivalence stated in
Theorem~\ref{thm4introd}.

Assume now that $D(M) \neq \mbox{\O}$. Hence, $M$ does not have
finite total curvature and the arguments in the
last paragraph show that $D_1(M)\neq
\mbox{\O }$. To conclude the proof of
item~{1} of the theorem it remains to analyze the
topology of $D_1(M)$.

We will now define a metric space structure on $D_1(M)$ which
generates a topology that coincides with the
topology of uniform $C^k$-convergence on compact
subsets of $\R^3$ for any $k\in \N$ (in particular,
compactness of $D_1(M)$ will follow from sequential
compactness). To do this, we first prove that
there exists some $\ve >0$ such that
$\overline{\B }(\ve )$ intersects every surface
$\Sigma \in D_1(M)$ in a unique component which is a
 graphical disk over its projection to the tangent
space to $\Sigma $ at $\vec{0}$ and with gradient less than 1.
Otherwise, there exists a sequence
$\{ \Sigma _n\} _n\subset D_1(M)$ such that this
property fails in the ball $\overline{\B }(\frac{1}{n})$ for every $n\in \N $.
As the Gaussian curvature of $\Sigma _n$ is not greater than $1$,
the uniform graph lemma~\cite{pro2} implies that around every point $p_n\in \Sigma _n$, this
surface can be locally expressed as a graph over a disk in the tangent space $T_{p_n}\Sigma _n$
of uniform radius. Therefore, there exists $\de >0$ such that for $n$ large, $\Sigma _n$ intersects
 $\overline{\B }(\de )$ in at least two components,
 one of which passes through~$\vec{0}$ and the other
one intersects $\B (\frac{1}{n})$, and such that both
components are graphical over domains in the tangent
space of $\Sigma _n$ at $\vec{0}$ with small gradient.
Hence, a subsequence of these $\Sigma _n$ (denoted in
the same way) converges smoothly to a minimal lamination ${\cal L}_1$ of
$\R^3$ with a leaf $L_1\in {\cal L}_1$ passing through $\vec{0}$ such that
the multiplicity of the limit $\{ \Sigma _n\} _n\to L_1$ is greater than
one. This last property implies that the two-sided cover $\wh{L}_1$
of $L_1$ is stable, by Lemma~\ref{lemma4.2}. As $\wh{L}_1$ is
complete and stable then $\wh{L}_1$ is a plane (and $L_1=\wh{L}_1$), which contradicts
that the convergence $\{ \Sigma _n\} _n\to L_1$ is smooth
and the curvature of the $\Sigma _n$ is $-1$ at
 $\vec{0}$ for every $n$. This proves our claim on the
existence of $\ve $.

With the above $\ve >0$ at hand, we define the distance between
any two surfaces $\Sigma _1,\Sigma _2\in D_1(M)$ as
\[
d(\Sigma _1,\Sigma _2)=d_{\cal H}\left( \Sigma _1\cap
\overline{\B}(\ve /2),\Sigma _2\cap \overline{\B}(\ve /2)\right) ,
\]
where $d_{\cal H}$ denotes the Hausdorff distance.
Standard elliptic theory implies that the metric
topology on $D_1(M)$ associated to the distance $d$
agrees with the topology of the uniform
$C^k$-convergence on compact sets of $\rth$ for any
$k$.

Next we prove that $D_1(M)$ is sequentially compact (hence compact).
Every sequence $\{ \Sigma _n\} _n\subset D_1(M)$ contains a subsequence
which converges to a minimal lamination ${\cal L}$ of
$\R^3$ with bounded Gaussian curvature $K_{\cal L}$
and $K_{\cal L}(\vec{0})=-1$. The same arguments given in the
second paragraph of this proof imply that ${\cal L}$ consists just of
the leaf $L$ passing through~$\vec{0}$, which is a
properly embedded minimal surface in $\R^3$.  Clearly
$L\in D_1(M)$, which proves item~{1} of the
theorem.

We now prove item~{2.} Using the definition of $D$-invariance,
it is elementary to show that for any
$\Sigma \in D(M)$, $D(\Sigma )$ is a closed set in
$D(M)$; essentially, this is because the set of limit
points of a set in a topological space forms a closed
set. That $D(\Sigma )$ is $D$-invariant follows from
property {\bf (i)} stated at the beginning of this section.
Similar techniques show that if $\Delta \subset D(M)$ is a $D$-invariant
subset, then its closure in $D(M)$ is also
$D$-invariant, and item~{2} of the theorem is proved.

Now assume that $\Delta $ is a minimal $D$-invariant
set in $D(M)$. If $\Delta $ contains a surface of
finite total curvature, then the minimality of
$\Delta $ implies $\Delta $ consists only of this
surface, and so, it is connected and closed in $D(M)$.
Suppose now that $\Delta$ does not consist of exactly
one surface of finite total curvature, or equivalently,
$\Delta $ contains no surfaces of finite total
curvature. Then, property {\bf (v)} above implies that
for any $\Sigma \in \Delta $, $D(\Sigma ) = \Delta $.
Since $D(\Sigma )$ is closed by item~{2,}
then $\Delta $ is closed as well. Since $D(\Sigma) =
\Delta$, then $\Delta$ also contains the path connected
subset $S\subset D(M)$ of all dilations of $\Sigma$.
Since $\Delta $ is a closed set in $D(M)$, then the closure of
$S$ in $D(M)$ is contained in $\Delta $. Reciprocally,
if $\Sigma _1\in \Delta =D(\Sigma )$ then $\Sigma _1$ is a
non-flat, properly embedded minimal surface in $\R^3$ which is the
$C^2$-limit of a sequence $\mu _n(\Sigma -p_n)$ for some
$\{ \mu _n\} _n\subset \R^+$ and $\{ p_n\} _n\subset \R^3$,
$p_n\to \infty $. As $\mu _n(\Sigma -p_n)\in S$ for each $n\in \N$, then
$\Sigma _1$ lies in the closure of $S$ in $D(M)$, and so,
$\Delta $ equals the closure of $S$ in $D(M)$ (in particular,
$\Delta $ is connected as $S$ is path connected). This proves
item~3 of the theorem.

Next we prove item~{4}. Suppose
$\Delta \subset D(M)$ is a non-empty, $D$-invariant set.
One possibility is that $\Delta $ contains a
surface $\Sigma $ of finite total curvature. By the
main statement of this theorem, $D(\Sigma )=\mbox{\O }$
and by property {\bf (ii)} above, $\Sigma $ is a minimal
element in $\Delta $. Now assume $\Delta $
contains no surfaces of finite total curvature.
Consider the  collection
\[
\Lambda =\{ \Delta '\subset \Delta \ | \ \Delta '\mbox{ is
non-empty, closed and $D$-invariant}\} .
\]
Note that $\Lambda $ is non-empty, since for any
$\Sigma \in \Delta$, the set $D(\Sigma) \subset \Delta$
is such a closed, non-empty $D$-invariant set by the
first statement in item~{2}.
$\Lambda $ has a partial ordering induced by inclusion.
We just need to check that any linearly ordered subset in
$\Lambda $ has a lower bound, and then apply Zorn's
Lemma to obtain item~{4} of the theorem. Suppose
$\Lambda '\subset \Lambda $ is a non-empty linearly
ordered subset. We must check that the intersection
$\bigcap _{\Delta'\in \Lambda '}\Delta '$ is an element
of $\Lambda $. In our case, this means we only need to prove
that such an intersection is non-empty, because the
 intersection of closed (resp. $D$-invariant)
sets is closed (resp. $D$-invariant).

Given $\Delta '\in \Lambda '$, consider the collection
of surfaces $\Delta'_1=\{ \Sigma \in
\Delta'\ | \ \vec{0}\in \Sigma ,\ |K_{\Sigma } |\leq 1,\ |K_{\Sigma
}|(\vec{0})=1\}$. Note that $\Delta _1'$ is a closed subset of
$D(M)$, since $\Delta'$ and $D_{1}(M)$ are closed in $D(M)$. The set
$\Delta'_{1}$ is non-empty by the following argument. Let $\Sigma
\in \Delta '$. Since $\Sigma $ does not have finite total curvature
and $\Delta '$ is $D$-invariant, $D(\Sigma )$ is a non-empty subset
of $\Delta '$. By item~{1}, $D_1(\Sigma )$ is a non-empty subset
of $\Delta '_1$, and so, $\Delta '_1$ is non-empty. Now define
$\Lambda_1 '=\{ \Delta_1'\ | \ \Delta ' \in
\Lambda '\} $. As $ \bigcap _{\Delta _1'\in \Lambda _1'}\Delta _1'=
\bigcap _{\Delta '\in \Lambda '}\Delta _1'= \bigcap _{\Delta '\in
\Lambda '}\left( \Delta '\cap D_1(M)\right) = \left( \bigcap
_{\Delta '\in \Lambda '}\Delta '\right) \cap D_1(M)$, in order to
check that $\bigcap _{\Delta '\in \Lambda '}\Delta '$ is non-empty,
it suffices to show that $\bigcap _{\Delta '_1\in \Lambda '_1}\Delta
'_1$ is non-empty. But this is clear since each element of
$\Lambda_1'$ is a closed subset of the compact metric space
$D_1(M)$, and the finite intersection property holds for the
collection $\Lambda_1'$.

Next we prove item~{5}. Let $\Delta \subset D(M)$
be a non-empty $D$-invariant subset which contains no surfaces of finite total
curvature. By item~{4}, there exists a minimal element $\Sigma
\in \Delta $. Since none of the surfaces of $\Delta $ have finite
total curvature, it follows that $D(\Sigma )\neq \mbox{\O }$. As
$\Sigma $ is a minimal element, there exists a minimal $D$-invariant
subset $\Delta'\subset D(M)$ such that $\Sigma \in \Delta'$. By property~{\bf (v)}
above, $D(\Sigma )=\Delta'$. Note that $\Delta_1'=\Delta '
\cap D_1(M)$ contains $D_1(\Sigma )$, which is non-empty since $D(\Sigma )\neq
\mbox{\O }$ (by item~{1} of this theorem). Then there exists a
surface $\Sigma _1\in \Delta_1'$, which in particular is a minimal
element (any element of $\Delta'$ is), and lies in $\Delta_1$
(because $\Delta'\subset \Delta $ by property {\bf (vi)}). Finally,
$\Sigma _1$ is dilation-periodic by property {\bf (v)}, thereby proving
item {5} of the theorem.

To prove item~{6}, suppose $\Sigma \in D(M)$ is a
minimal element with finite genus, and assume
also that $\Sigma $ has infinite total curvature.
Since $\Sigma \in D(\Sigma )$ (by property {\bf (v)}
applied to $\Delta =D(\Sigma )$) and $\Sigma $ has
finite genus, then the genus of $\Sigma $ must be zero.
In this setting, the classification results
in~\cite{col1,mr8,mpr6} imply that $\Sigma $ is a
helicoid or a Riemann minimal example.

Finally we prove item~{7.} Consider a minimal
 element $\Sigma $ of $D(M)$ with more than one end.
 By the Ordering Theorem~\cite{fme2}, after possibly
 a rotation in $\R^3$ so that the limit tangent plane
at infinity for $\Sigma $ is horizontal (see~\cite{chm3}
for this notion of limit tangent
plane at infinity), $\Sigma $ has a
 middle end $e$ in the natural ordering of the ends
  of $\Sigma $ by their relative heights.
The results in~\cite{ckmr1} imply
that $e$ is a simple end and $e$ admits an end
representative $E\subset \Sigma $ with the following
properties:
\begin{itemize}
\item $E$ is a proper, non-compact subdomain of
$\Sigma $ with compact boundary and only one end.
\item $E$ is contained in the open region
$W\subset \R^3-\B (R)$ between two graphical, disjoint
vertical half-catenoids or horizontal planes, where $R>0$
is sufficiently large. Furthermore, we can
assume that $\Sigma \cap W=E$.
\end{itemize}
We now discuss two cases for $E$. First assume that $E$ has
quadratic decay of curvature. By Theorem~\ref{tt2.4}, $E$ has
finite total curvature, in which case item~{7} is known to hold.
So, assume that $E$ has infinite total curvature and we will find a
contradiction. On one hand, the proof of the first statement of
this theorem applies to $E$ to produce the collection
$D(E)$ of properly embedded minimal surfaces in $\R^3$ which are limits of $E$ under a
sequence of dilations with divergent translational part (note that
surfaces in $D(E)$ have empty boundary). Notice that under every
divergent sequence of dilations which give rise to a surface in
$D(E)$, the dilated regions $W_n$ related to $W$ converge to all of
$\R^3$, by the Half-space Theorem. This implies that
$D(E)$ is a subset of $D(\Sigma )$. Since $\Sigma $ is a minimal
element, then $D(E)=D(\Sigma )$. On the other hand, the results in~\cite{ckmr1} imply
that $E$ has quadratic area growth, and hence, the monotonicity
formula gives that every surface in $D(E)$ has at most the same
quadratic area growth as $E$. This discussion applies to $\Sigma $
since $\Sigma \in D(\Sigma )$. But since $\Sigma $ has other ends
different from $E$, the quadratic area growth of $\Sigma$ is
strictly greater than the one of $E$. This contradiction proves
item~{7,} thereby finishing the proof of Theorem~\ref{thm4introd}.
{\hfill\penalty10000\raisebox{-.09em}{$\Box$}\par\medskip}
\par
\vspace{.2cm}

\section{Internal dynamical periodicity of properly embedded minimal surfaces with
infinite total curvature.}
\label{sec:6}

Let $M\subset \R^3$ be a properly embedded minimal surface with infinite
total curvature and let $\Sigma \in D(M)$ be a minimal element, which exists by
item~4 of Theorem~\ref{thm4introd}. Assume that $\Sigma $ also has infinite
total curvature. We claim that each compact subdomain of $\Sigma $ can be approximated with
arbitrarily high precision (under dilation) by an infinite collection
of pairwise disjoint compact subdomains of $\Sigma $, and these approximations
can be chosen not too far from each other.
This property can be considered to be a weak notion
of periodicity; next we describe a more
precise statement for this `weak periodicity' property for $\Sigma $.

As $\Sigma \in D(M)$ is a minimal element, there exists a minimal $D$-invariant
subset $\Delta \subset D(M)$ such that $\Sigma\in \Delta $. Since
$\Sigma $ has infinite total curvature, then $D(\Sigma )=\Delta $, in particular
$\Sigma \in D(\Sigma )$ (i.e., $\Sigma $ is quasi-dilation-periodic).
Given $R > 0$ such that $\Sigma $ intersects
$\esf^2(R)$ transversely, let $\Sigma (R)=\Sigma \cap \overline{\B}(R)$.
Since $\Sigma \in D(\Sigma )$, then for every small $\ve>0$ there
exists a collection $\{ \overline{\B }_n = \overline{\B }(p_n, R_n) \}_n$
of disjoint closed balls such that the surfaces $\Sigma_n =
\frac{R}{R_n} ((\Sigma \cap \overline{\B }_n) -p_n)$ can be parameterized by
$\Sigma (R)$ in such a way that as mappings, they are $\ve$-close
to $\Sigma(R)$ in the $C^2$-norm. By Zorn's lemma, any such
collection of balls $\{ \overline{\B }_n\} _n$ is contained in a maximal collection
Max$(\Sigma ,R,\ve )$ of closed round balls $\overline{\B }=\overline{\B }(p,r)$ so that
$\frac{R}{r}[(\Sigma
\cap \overline{\B})-p]$ can be parameterized by $\Sigma (R)$ in such a way that as mappings, they are
$\ve $-close in the $C^2$-norm. After denoting the elements of Max$(\Sigma ,R,\ve )$
by $\overline{\B }_n$, $n\in \N$, we define for every $n\in \N$ the positive number
\begin{equation}
\label{eq:d(n)}
d(n;R,\ve )=\inf _{m\neq n}\left\{ \frac{1}{R_n}\mbox{dist}_{\R^3}(\overline{\B }_n,
\overline{\B }_m)\ | \ \overline{\B}_m\in \mbox{Max}(\Sigma ,R,\ve ),\  m\neq n\right\} ,
\end{equation}
where $\overline{\B }_n,\overline{\B }_m\in \mbox{Max}(\Sigma ,R,\ve )$.
Hence, $d(n;R,\ve )$ measures the minimum relative distance from the ball
$\overline{\B }_n\in \mbox{Max}(\Sigma ,R,\ve )$
to any other ball $\overline{\B }_m$ in the collection
Max$(\Sigma ,R,\ve )$, where by {\it relative} we mean that
$\overline{\B }_n$ is normalized to have
radius 1; this is the task of dividing by $R_n$ in (\ref{eq:d(n)}).
In this situation we will prove the following property, which
expresses in a precise way the `weak periodicity' mentioned in the first paragraph
of this section.

\begin{proposition}
\label{propos4.1}
If $\Sigma (R)$ is not an $\frac{\ve }{2}$-graph over
a disk in a plane, then the sequence \newline
 $\{ d(n;R,\ve )\} _n\subset (0,\infty )$
is bounded from above (the bound depends on $\Sigma, R,\ve $).
\end{proposition}
\begin{proof}
We argue by contradiction. Otherwise,
there exists a sequence of integers
$\{ n(i)\} _{i\in \N }$ such that $d(n(i);R,\ve)\geq i$ for all $i\in \N$.
We define $\widehat{\Sigma }(i)=\frac{R}{R_{n(i)}}
(\Sigma -p_{n(i)})$, $i\in \N $. Observe that the following two properties hold:
\begin{description}
\item[{\rm (P1)}] The Gaussian
curvature of $\{ \widehat{\Sigma }(i)\cap \overline{\B }(R)\} _i$
is uniformly bounded (by the defining properties of the balls in the family
$\mbox{Max}(\Sigma ,R,\ve )$).
\item[{\rm (P2)}] Given
$\overline{\B }_m\in \mbox{Max}(\Sigma ,R,\ve )$ with $m\neq n(i)$,
the closed ball $\overline{\B }'_m=\frac{R}{R_{n(i)}}
(\overline{\B }_m-p_{n(i)})$ is at distance at least $iR$ from $\overline{\B }(R)$.
\end{description}
We claim that the sequence of surfaces
$\widehat{\Sigma }(i)$ has locally bounded Gaussian
 curvature outside of $\overline{\B }(R)$: otherwise,
the Gaussian curvature $K_{\widehat{\Sigma }(i)}$ blows-up
around a point $P\in \R^3-\overline{\B }(R)$, and in this case
we can blow-up $\widehat{\Sigma }(i)$ on the scale of curvature
around points of almost-maximal curvature tending to $P$,
thereby obtaining a new limit surface $\Sigma '$ of rescaled copies
of portions of $\widehat{\Sigma }(i)$ in small extrinsic balls around $P$
(in particular, these small balls are disjoint from the balls corresponding
to elements in $\mbox{Max}(\Sigma ,R,\ve )$).
Hence, $\Sigma '$ lies in $D(\Sigma )$; since $\Sigma $ is a minimal element, then
$\Sigma \in D(\Sigma ')$, which contradicts the maximality of the
collection Max$(\Sigma ,R,\ve )$. This contradiction proves our claim.

As $\{ \widehat{\Sigma }(i)\cap [\R^3 -\overline{\B}(R)]\} _i$ has locally bounded Gaussian
 curvature, we conclude that after replacing by a subsequence,
  $\{ \widehat{\Sigma }(i)\cap [\R^3-
 \overline{\B }(2R)]\} _i$ converges as $n\to \infty $ to a minimal lamination
 ${\cal L}'(2R)$ of $\R^3-\overline{\B }(2R)$.
 This lamination ${\cal L}'(2R)$ has quadratic decay
 of curvature (otherwise we again contradict the
 minimality of $\Sigma $ and the maximality of Max$(\Sigma ,R,\ve )$ as before), and so, ${\cal L}'(2R)$
 has bounded curvature in $\R^3-\B (3R)$.

 Next we claim that the surfaces $\widehat{\Sigma }(i)\cap
\overline{\B }(3R)$ do not have uniformly bounded
curvature. Arguing again by contradiction, the failure of our claim together with
the arguments in the last paragraph imply that
 $\{ \widehat{\Sigma }(i)\} _i$ converges to a lamination
 ${\cal L}'$ of $\R^3$ with quadratic curvature decay.
Since we are assuming that $\Sigma (R)$ is not
an $\frac{\ve }{2}$-graph over a flat disk, then
${\cal L}'$ cannot be flat. Hence,
Proposition~\ref{corollamin} implies that
${\cal L}'$ consists of a single leaf which is
a properly embedded minimal surface $\Sigma '$
with finite total curvature. As before, this implies that
$\Sigma '\in D(\Sigma )$. This is absurd, since then
$\Sigma \in D(\Sigma ')=\mbox{\O }$. This contradiction shows that
the $\widehat{\Sigma }(i)\cap
\overline{\B }(3R)$ do not have uniformly bounded
curvature.

Therefore, after extracting a subsequence we can find for each $i\in \N$ a point
$\widehat{p}(i)\in \widehat{\Sigma }(i)\cap
\overline{\B }(3R)$ such that
$|K_{\widehat{\Sigma }(i)}|(\widehat{p}(i))\to
\infty $ as $i\to \infty $. We can also assume
that $|K_{\widehat{\Sigma }(i)}|$ attains its
maximum value in the compact set $\widehat{\Sigma }(i)\cap \overline{\B }(3R)$ at
$\widehat{p}(i)$, for all $i$. After translating by
$-\widehat{p}(i)$, rescaling by
$\sqrt{|K_{\widehat{\Sigma }(i)}|(\widehat{p}(i))}$
and extracting another subsequence, we obtain a new limit
surface $\Sigma ''\in D_1(\Sigma )$. We now have two possibilities,
depending on whether or not the sequence
of open balls $\{ \sqrt{|K_{\widehat{\Sigma }(i)}|(\widehat{p}(i))}
(\B (R)-\widehat{p}(i))\} _i$ eventually leaves every compact
set of $\R^3$.

Firstly suppose that the sequence
of open balls $\{ \sqrt{|K_{\widehat{\Sigma }(i)}|(\widehat{p}(i))}
(\B (R)-\widehat{p}(i))\} _i$ fails to leave every compact
set of $\R^3$. Then, after choosing a subsequence these balls converge
to a closed halfspace $H^+$ and by property (P1) above,
$\Sigma ''\cap H^+$ is flat. As $\Sigma ''$
is not flat, then we conclude that $\Sigma ''$ cannot
intersect $H^+$, or equivalently $\Sigma ''$ is contained
in a halfspace. This contradicts the fact that the
Gaussian curvature of the non-flat surface
$\Sigma ''$ is bounded, see~\cite{mr8}. Therefore,
the sequence of balls $\{ \sqrt{|K_{\widehat{\Sigma }(i)}|(\widehat{p}(i))}
(\B (R)-\widehat{p}(i))\} _i$ leaves every compact
set of $\R^3$. Since
$\Sigma ''\in D(\Sigma )$ and $\Sigma $ is a minimal element,
then $\Sigma \in D(\Sigma '')$, which by the same arguments as before
contradicts the maximality of the collection Max$(\Sigma ,R,
\ve )$. Now the proposition is proved.
\end{proof}

\begin{remark} \label{almost-finite} {\rm
We remark that in general there exist properly
embedded minimal surfaces $M$ with infinite total curvature,
such that $D(M)$ contains more than one minimal
$D$-invariant set, see some examples in the discussion before Theorem~11.0.13
in~\cite{mpe10}.
}
\end{remark}

\center{William H. Meeks, III at bill@math.umass.edu\\
Mathematics Department, University of Massachusetts, Amherst, MA 01003}
\center{Joaqu\'\i n P\'{e}rez at jperez@ugr.es \qquad\qquad Antonio Ros at aros@ugr.es\\
Department of Geometry and Topology, University of Granada, Granada, Spain}
%\center{Antonio Ros at aros@ugr.es\\
%Department of Geometry and Topology, University of Granada, Granada, Spain}

\end{document}